\documentclass[11pt, reqno]{amsart}
\usepackage{amssymb,latexsym,amsmath,amsfonts}
\usepackage[mathscr]{eucal}
	
\numberwithin{equation}{section}

\theoremstyle{definition}
\newtheorem{definition}{Definition}[section]

\newtheorem{remark}[definition]{Remark}

\theoremstyle{plain}
\newtheorem{theorem}[definition]{Theorem}
\newtheorem{lemma}[definition]{Lemma}

\newtheorem{result}[definition]{Result}
\newtheorem{obs}[definition]{Observation}

\newcommand{\eps}{\varepsilon}

\newcommand{\zt}{\zeta}
\newcommand{\zbar}{\overline{z}}

\newcommand{\lbar}{\overline{\lambda}}
\newcommand{\albar}{\overline{\alpha}}
\newcommand{\al}{\alpha}
\newcommand{\lm}{\lambda}
\newcommand{\tht}{\theta}

\newcommand \del[1]{\delta_{#1}}
\newcommand{\dl}{\delta}
\newcommand\ba[1]{\overline{#1}}
\newcommand\hull[1]{\widehat{#1}}
\newcommand\wtil[1]{\widetilde{#1}}
\newcommand{\id}{\mathbb{I}}
\newcommand{\bt}{\beta}
\newcommand{\sm}{\sigma}
\newcommand{\om}{\omega}

\newcommand{\bdy}{\partial}

\newcommand{\smoo}{\mathcal{C}}

\newcommand{\bcdot}{\boldsymbol{\cdot}}
\newcommand{\rl}{\Re\mathfrak{e}}
\newcommand{\imag}{\Im\mathfrak{m}}
\newcommand{\impl}{\Longrightarrow}
\newcommand{\mapp}{\longrightarrow}

\newcommand{\CC}{\mathbb{C}^2}
\newcommand{\cplx}{\mathbb{C}}
\newcommand{\RR}{\mathbb{R}^2}
\newcommand{\rea}{\mathbb{R}}

\newcommand{\upl}{M_{1} \cup M_{2}}
\newcommand{\usf}{S_1 \cup S_2}
\newcommand\utgt[2]{T_0{#1}\cup T_0{#2}}

\begin{document}

\title[Local polynomial convexity of union of surfaces]{Local polynomial convexity of the union
of two \\ totally-real surfaces at their intersection}
\author{Sushil Gorai}
\address{Department of Mathematics, Indian Institute of Science, Bangalore -- 560 012}
\email{sushil@math.iisc.ernet.in}
\thanks{This work is supported by CSIR-UGC fellowship 09/079(2063) and by the UGC under DSA-SAP,
Phase IV}

\begin{abstract}
We consider the following question: Let $S_1$ and $S_2$ be two smooth, totally-real
surfaces in $\CC$ that contain the origin. If the union of their tangent planes is
locally polynomially convex at the origin, then is $S_1 \cup S_2$ locally polynomially convex
at the origin? If $T_0S_1 \cap T_0S_2=\{0\}$, then it is a folk result that the answer is
{\em yes}. We discuss an obstruction to the presumed proof, and provide a different approach. 
When $dim_\rea (T_0S_1 \cap T_0S_2)=1$, we present a geometric condition under which no consistent
answer to the above question exists. We then discuss conditions under which we can expect local polynomial convexity.

\end{abstract}
\maketitle

\section{Introduction and Statement of Results}\label{S:intro}

The aim of this paper is to provide an answer to the following question:
\smallskip

\begin{itemize}
\item[$(*)$]{\em Let $S_1$ and $S_2$ be two smooth, totally-real surfaces in $\CC$ that contain the
origin. If the union of their tangent planes is locally polynomially convex at the origin, then is
$S_1 \cup S_2$ locally polynomially convex at the origin? }
\end{itemize}
\smallskip

Our interest is to provide a {\bf complete} analysis of the situation. We were motivated by the 
following circumstances --- which will
explain our emphasis on the word ``complete'' ---  to discuss the question $(*)$.
\begin{itemize}
\item[1)] Let $S_1$ and $S_2$ be as above. When $T_0S_1 \cap T_0S_2=\{0\}$, the problem is 
no doubt familiar to the experts. In this case, the answer to $(*)$ is expected to be in the
affirmative. The proof, it is asserted, follows from a slight modification of an argument given
by Forstneri{\v{c}} and Stout in \cite{FS}. While this will work {\em for most pairs $(S_1,S_2)$ 
in $\CC$} (in a sense that will be explained below) it is not clear if such an approach will 
work universally. {\em The reader is urged to look at the discussion that immediately 
follows this list.}
\smallskip

\item[2)] It turns out that when $T_0S_1$ and $T_0S_2$ contain a line, then $\utgt{S_1}{S_2}$ is always
locally polynomially convex at the origin. There are some partial answers to $(*)$ when 
$dim_\rea(T_0S_1 \cap T_0S_2)=1$; see, for instance, \cite{D}. However, many of the results that
we are aware of require $S_1$ and $S_2$ to be {\em real-analytic} surfaces (and one of these
results contains an error; see Remark \ref{Rem:dieu}). In contrast, we wish to answer $(*)$ when $S_1$ and $S_2$ are
merely $\smoo^k$-smooth, $k\geq 2$.    
\smallskip

\item[3)] It turns out that, under a certain natural geometric condition, there is no consistent answer 
to $(*)$ when $dim_\rea(T_0S_1 \cap T_0S_2)=1$. We would like to demonstrate rigorously what this means, 
and also to give some conditions under which $\usf$ is locally polynomially convex at the origin.
\end{itemize}
\smallskip

Let us first consider $(*)$ in the case when $T_0S_1 \cap T_0S_2=\{0\}$. It has been asserted that the
proof of the fact that the answer to $(*)$ is, ``Yes,'' is implicit in \cite{FS}. Such a proof would go 
as follows:
\begin{itemize}
\item {\em Step 1.} Show that there is an invertible $\cplx$-linear transformation that transforms
$\utgt{S_1}{S_2}$ to $M_1 \cup M_2$, where $M_1$ and $M_2$ are totally-real planes of the form
\[
(**) \quad
	\begin{cases}
		M_1 \ : \ w = \zbar & {} \\
		M_2 \ : \ w = r\zbar + \varrho z, & 
		r\neq 0, \ (r,\varrho)\in \RR\setminus\{(1,0)\} \qquad\qquad\qquad
		\end{cases}
\]
\item {\em Step 2.} Use the fact that $\utgt{S_1}{S_2}$ is locally polynomially convex at $0$ 
and apply Kallin's Lemma in a similar manner as in \cite{FS} to infer that $S_1 \cup S_2$ is
locally polynomially convex at $0$.
\end{itemize}
The reason we require $M_1$ and $M_2$ to have the form $(**)$ is because {\em there seems to be no
simple way to deduce the desired result via Kallin's Lemma unless $r$ and $\varrho$ in $(**)$ are
real}. While the transformation described in Step~1 is possible for most pairs of transverse totally 
real planes (whose union is locally polynomially convex at the origin) representing $(T_0S_1,T_0S_2)$,  
we must also contend with the following:
\smallskip 

\begin{obs}\label{O:counterEx} There is at least one one-parameter family of linear
transformations $\{S_p:p\in \rea\setminus\{0\}\}$ of $\CC$ such that
\begin{align}
 (S_p+i\id)(\RR) \ \text{is totally real} \; &\forall p\in \rea\setminus\{0\}, \notag \\
 (S_p+i\id)(\RR)\cap \RR = \{0\} \; &\forall p\in \rea\setminus\{0\}, \notag \\
 (S_p+i\id)(\RR)\cup \RR \ \text{is locally polynomially convex at $0$} \; 
 &\forall p\in \rea\setminus\{0\}, \notag
\end{align}
but for each $p\in\rea\setminus\{0\}$, there exists no invertible $\cplx$-linear transformation 
of $\CC$ that can map $\RR \cup (S_p+i\id)(\RR)$ to a union $M_1 \cup M_2$ with 
$(M_1,M_2)$ having the form $(**)$.
\end{obs} 
\smallskip

The details of the above are presented in sub-section~\ref{SS:counterEx}. We do not doubt
that the above two-step approach {\em could be made to work even when $r$ and $\varrho$ in $(**)$ 
take non-real values}, but this would require at least a more sophisticated Kallin polynomial
(hence much harder calculations) and may, perhaps, even require some further inputs besides those in \cite{FS}. 
Consequently, we try another approach by modifying some ideas of Weinstock --- which enables us
to deal with $\utgt{S_1}{S_2}$ without having to transform the planes to graphs --- to get
Theorem~\ref{T:transversePCVX} below. The latter method has the advantage that it is more
readily adapted to the problem of studying local polynomial convexity at $0\in \CC$ of the 
union of {\em more than two} totally-real planes in $\CC$ intersecting at $0$. The latter
problem is of some interest because it provides the means to investigate local polynomial 
convexity of a smooth real surface $S\subset \CC$ at a point $p\in S$ at which  $T_pS$ is a
complex line. This general principle was, in fact, introduced in \cite{FS}. In general --- as
the papers \cite{GB1} and \cite{GB2} reveal --- detecting local polynomial convexity 
at a degenerate ``non-parabolic'' complex-tangency would require the study of the union of more
than two totally-real surfaces, intersecting transversely at $0\in \CC$. These issues will be 
tackled in a different article. With this background, we can announce:

\begin{theorem}\label{T:transversePCVX}
 The union of two $\smoo^2$-smooth totally-real surfaces
 in $\mathbb{C}^2$ intersecting transversally only at the origin is locally polynomially convex if
 the union of their tangent spaces at the origin is locally polynomially convex at the origin.
\end{theorem}

Weinstock \cite{Wk} gave a criterion for the union of two transverse, maximally totally-real
subspaces in $\cplx^n$ to be locally polynomially convex at the origin. Our proof of
Theorem~\ref{T:transversePCVX} relies upon a normal form, developed in \cite{Wk}, for a pair of 
totally-real planes intersecting transversely at $0\in \CC$, and on Kallin's Lemma (see
Lemma~\ref{L:kallin} below). We note that the condition stated in $(*)$ {\em cannot} be 
necessary and sufficient; see \cite[Example 5]{Wk}.
\smallskip

The following lemma is essential in setting the context for the next three theorems.

\begin{lemma}\label{L:arcPCVX}
 Let $M_j, \ j=1,2,$ be two distinct totally-real planes in $\CC$ containing the origin, such that
 $dim_\rea(M_1 \cap M_2)=1$. Then $\upl$ is locally polynomially convex at the origin.
\end{lemma}

\noindent{We shall {\em not} give a separate proof for the above; the proof will follow along
similar lines as the proof of Theorem~\ref{T:linePcvxFlat} below. This lemma establishes that question 
$(*)$ remains valid when $dim_\rea(M_1 \cap M_2)=1$. Our next theorem shows that the answer to $(*)$ 
is not always affirmative when $dim_\rea (T_0S_1 \cap T_0S_2)=1$. Comparing this with 
Theorem \ref{T:linePCVX} will reveal that that there is no consistent answer to $(*)$ when 
\begin{itemize}
 \item[(I)] $dim_\rea (T_0S_1 \cap T_0S_2)=1$; and
 \item[(II)] $span_\cplx\{T_0S_1\cap T_0S_2\}\subset \ span_{\rea}\{\utgt{S_1}{S_2}\}$.
\end{itemize}
Refer to the remarks following Theorem~\ref{T:linePCVX} for a clarification of the last assertion.
\smallskip

Before stating Theorem \ref{T:epsperturb}, we need to define one term. Given a set 
$S \subset \CC$, we say that $S^\eps$ is an {\em$\eps$-perturbation} of $S$ if $S^\eps$ is 
the image of $S$ under a $\smoo^1$-diffeomorphism  $\Theta_\eps$ defined in a neighbourhood 
$U$ of $S$ such that $\|{\rm id}_U - \Theta_\eps \|_{\smoo^1} \lesssim \eps$.

\begin{theorem}\label{T:epsperturb}
 Let $M_j,~j=1,2$, be two distinct totally-real planes in $\CC$ containing the origin, such that
 $dim_\rea(M_1 \cap M_2)=1$ and $span_\cplx\{M_1 \cap M_2\}$ lies in the real hyperspace that contains
 $\upl$. Then for each $\eps >0$, there exists a $\dl>0$ and totally-real submanifolds $S_j^\eps,~j=1,2,$
 of $B(0;2\dl)$ such that:
 \begin{itemize}
  \item $S_j^\eps \cap \ba{B(0;\dl)}$ are $\eps$-perturbations of $M_j \cap \ba{B(0;\dl)},~~j=1,2,$
  \item $T_0S_1^\eps \cup T_0S_2^\eps = \upl$,
 \end{itemize}
 and such that $S_1^\eps \cup S_2^\eps$ is not polynomially convex.
\end{theorem}
\noindent Here, $B(a;r)$ denotes the Euclidean ball in $\CC$ with centre at $a$ and radius $r>0$.
\smallskip 

Theorem \ref{T:epsperturb} raises the following question: {\em what can we say if $S_1$ and 
$S_2$ are as in $(*)$, $dim_\rea (T_0S_1 \cap T_0S_2)=1$, and 
$span_\cplx\{T_0S_1\cap T_0S_2\}\nsubseteq \ span_{\rea}\{\utgt{S_1}{S_2}\}$~?} In response to 
this question we have the following result:

\begin{theorem}\label{T:linePcvxFlat}
 Let $S_1$ and $S_2$ be two $\smoo^2$-smooth surfaces in $\CC$ that contain the origin. Assume that
 \begin{itemize}
  \item $dim_\rea (T_0S_1 \cap T_0S_2)=1$; and
  \item $span_\cplx\{T_0S_1\cap T_0S_2\}\nsubseteq \ span_{\rea}\{\utgt{S_1}{S_2}\}$.
 \end{itemize}
 If $(S_1\cup S_2)\subset \ span_{\rea}\{\utgt{S_1}{S_2}\}$, then $S_1\cup S_2$ is locally polynomially
 convex at the origin.
\end{theorem}

\begin{remark}\label{Rem:dieu}
Unbeknownst to me, Dieu had announced the following result in \cite{D}:
\begin{itemize}
\item[{}] \begin{result}[Prop.~2.2, \cite{D}]\label{R:dieuWrong}
Let $\varphi$ be a real-valued function defined in a neighbourhood of $0\in\cplx$ and of class
$\smoo^1$. Define
\begin{align}
 S_1 &:= \{(z,w)\in \CC: w=\zbar \}, \notag \\
 S_2 &:= \{(z,w)\in {\sf Dom}(\varphi)\times \cplx: w=(1+\lambda)\zbar + \ba{\lambda}z + \varphi(z)\}\; (\lm \neq 0,-1). \notag
\end{align}
Then $S_1\cup S_2$ is {\em not} locally polynomially convex at $0$ if and only if
\begin{itemize}
\item[i)] $\lambda$ is real; and
\item[ii)] For every $t$ sufficiently close to $0\in \rea$, the set $\{z \in {\sf Dom}(\varphi):\rl(z)=t/2, \;
2\lambda\rl(z)+\varphi(z)=0\}$ contains at most one component.
\end{itemize}
\end{result}
\end{itemize}
\noindent{It was brought to my notice that Theorem~\ref{T:linePcvxFlat} follows immediately from 
the above result; or --- at any rate --- in the generic arrangement of tangents when 
$T_0S_1=\{(z,w)\in \CC: w=\zbar \}$ and $T_0S_2=\{(z,w)\in \CC: w=(1+\lambda)\zbar + \ba{\lambda}z\}$
and when $S_1$ and $S_2$ are $\smoo^{\om}$-surfaces (in which case $S_1$ can always be taken as $\{(z,w)\;:\;w=\zbar\}$
locally). In this setting, the argument would go as follows:
\begin{itemize}
\item The condition in Result~\ref{R:dieuWrong} that $\varphi$ be real-valued is equivalent to
our condition $(S_1\cup S_2)\subset \ span_{\rea}\{\utgt{S_1}{S_2}\}$ (in Theorem~\ref{T:linePcvxFlat});
and
\item The {\em negation} of the condition (i) in Result~\ref{R:dieuWrong} is equivalent to our
condition $span_\cplx\{T_0S_1\cap T_0S_2\}\nsubseteq \ span_{\rea}\{\utgt{S_1}{S_2}\}$.
\end{itemize}
However, it turns out that {\em the condition $[(i)~\text{AND}~(ii)]$ is neither necessary nor sufficient
for $S_1\cup S_2$ to fail to be polynomially convex.} A demonstration of this is presented in
sub-section~\ref{SS:counterDieu} below. This observation also demands that we prove 
Theorem~\ref{T:linePcvxFlat} from scratch.} 
\end{remark}

We now consider totally-real graphs $S_j, \ j=1,2$. When $dim_\rea(T_0S_1 \cap T_0S_2)=1$,
then one expects polynomial convexity to be influenced by the higher-order terms in the graphing 
functions. This is the intuition behind the next theorem. Given such graphs, it can be shown that there
is a global holomorphic change of coordinates with respect to which $S_1$ and $S_2$ have the representations
given in Theorem \ref{T:linePCVX}. To reiterate: the representations of the graphs $S_1$ and $S_2$ in the 
first half of Theorem \ref{T:linePCVX} {\em are not simplifying assumptions}.

\begin{theorem}\label{T:linePCVX}
Let $S_j,~j=1,2$, be two $\smoo^\infty$-smooth totally-real surfaces in $\cplx^2$ containing the origin such that
$T_0S_1 \neq T_0S_2$ and $T_0S_1\cap T_0S_2$ contains a real line. In a neighbourhood $U$
of the origin, we present:
 \begin{align}
  S_{1}\cap U &= \{(z,\zbar+\overline{A} z^2 + A \zbar^2+C_1 z \zbar+ O(|z|^3) ): z \in D(0;\dl)\}, \notag \\
  S_{2}\cap U &= \{(z,\zbar+\overline{\lambda}z+\lambda\zbar +\phi_{2}(z)):
  z\in D(0;\dl)\}, \notag
 \end{align}
 where $\dl>0$, $\phi_2 \in \smoo^\infty(D(0;\dl))$ and
 $ \phi_2(z)= A_2 z^2 +B_2 \zbar^2+C_2 z \zbar+O(|z|^3).$\\
 \smallskip
 Suppose:
 \begin{enumerate}
 \item[$(i)$](Non-degeneracy condition) $\imag(C_1) \neq 0$,
  $\imag(\frac{(\overline{A_2}-B_2)\lbar^2}{|\lambda|^2}+ C_2)\neq 0$ and have opposite signs;
\item[$(ii)$] ${\rm sgn}(\imag(C_1))\imag\left((\overline{A_2}-B_2)\lbar\right)< 
\frac{1}{2}\left|\imag \left (\frac{(\overline{A_2}-B_2)\lbar^2}{|\lambda|^2}+ C_2\right)\right|$.
\end{enumerate}
 Then $\usf$ is locally polynomially convex at the origin.
\end{theorem}

\begin{remark}
The conditions $(i)$ and $(ii)$ might look somewhat artificial at first glance, but we formulated
them with the following phenomenon in mind. When $\ba{A_2}=B_2$ and $\imag{(C_j)}=0, ~j=1,2$, then the resulting
graphs $S_1^0$ and $S^2_0$ are in fact examples of the surfaces discussed in Theorem \ref{T:epsperturb}. Still keeping
$\ba{A_2}=B_2$, we see that if we alter the coefficients $C_j$ slightly so that $\imag{(Cj)}= \eps,~j=1,2$, then the
resulting $\usf$ is an $\eps$-perturbation of $S_1^0 \cup S_2^0$, and the local hull of $S_1^0 \cup S_2^0$ collapses
under the perturbation. Summarizing in a coordinate-free manner: when $dim_\rea(T_0S_1 \cap T_0S_2)=1$ and
\[
span_\cplx\{T_0S_1\cap T_0S_2\}\subset \ span_{\rea}\{\utgt{S_1}{S_2}\},
\]
it is possible for $S_1\cup S_2$ to not be locally polynomially convex at $0$ and yet, given any $\eps>0$, 
admit $\eps$-perturbations $S^\eps_j$ with $T_0S^\eps_j=T_0S_j, \ j=1,2$, such that $S_1^\eps \cup S_2^\eps$ 
is locally polynomially convex at the origin. I.e., when the pair $(S_1,S_2)$ has the properties (I) and (II) 
listed just after Lemma~\ref{L:arcPCVX}, then the question $(*)$ has 
no coherent answer.
\end{remark}

A few words about the layout of this paper. We would first like to conclude the technical discussion on
the relationship between a couple of theorems and the folk results to which they seem associated. This will
be the subject of the next section. Section~\ref{S:technical} will elaborate on some technical preliminaries
needed in the proofs of our results. The proofs of our four theorems will be found in Sections
\ref{S:proof-transversePCVX}--\ref{S:proof-linePCVX}.
\medskip

\section{Relations with known results}\label{S:relations}

\subsection{Concerning Observation~\ref{O:counterEx}}\label{SS:counterEx} Consider the two 
planes: $P_1:=\RR$ and $P_2:=span_{\rea}\{(s,t), (\sm,\tau)\}$ --- $s,t,\sm,\tau\in \cplx$
--- with $P_1\cap P_2=\{0\}$. First, note that if there exists a $\cplx-$linear, invertible map
$T: \CC\longrightarrow \CC$ such that
\begin{align}
 T(P_1) &= \{(z,w)\in \CC: w=\zbar \}, \notag\\
 T(P_2) &= \{(z,w)\in \CC: w=r\zbar + \varrho z \}, \;\; r\neq 0, \ (r,\varrho)\in \RR\setminus\{(1,0)\},
 \label{E:order1}
\end{align}  
then $T$ must have the matrix representation $M_T$ (with respect to the standard basis) given by
\[ M_T=
\begin{pmatrix}
 A & B \\
\ba{A} & \ba{B}
\end{pmatrix},
\]
where $A(=\al_1+i\al_2), \; B(=\bt_1+i\bt_2) \in \cplx$ and, for invertibility
$A\ba{B} \notin \rea$. If, however, we interchange the desired images of $P_1$ and $P_2$ under $T$, then
$T$ must have the following matrix representation:
\begin{equation} \label{E:mat2}
 M_T=
\begin{pmatrix}
 A & B \\
r\ba{A}+\varrho A & r\ba{B}+\varrho{B}
\end{pmatrix}.
\end{equation}
Motivated by Weinstock's work \cite{Wk}, we shall focus on $P_2(:=span_{\rea}\{(s,t), (\sm,\tau)\})$
determined by
\[
 \begin{pmatrix}
  s & \sm \\
  t & \tau   
 \end{pmatrix}
= \begin{pmatrix}
   p+i & 0 \\
   0 & q+i
  \end{pmatrix}
\]
(which gives one of the three normal forms for a pair of totally-real planes in $\CC$ intersecting
transversely at $0\in\CC$).
\smallskip

$T$ having the mapping properties given in \eqref{E:order1}
exists (and we will implicitly view the necessary conditions
as a linear system with $r$ and $\varrho$ as unknowns):
\begin{align}
 \Rightarrow \;\; &
\begin{cases}
\ba{A}(p-i)r+ A(p+i)\varrho &= \ \ba{A}(p+i)\\
\ba{B}(q-i)r+ B(q+i) \varrho &= \ \ba{B}(q+i)
\end{cases} \notag \\
& \quad \text{\em has a solution in} \; \RR\setminus ((\{0\}\times\rea)\cup \{(1,0)\}) \notag \\
& \quad \text{\em for some $(A,B)\in \CC$ such that $A\ba{B}\notin \rea$.} \notag
\end{align}
Considering real and imaginary parts separately, the existence of the desired $T$
\begin{align}
\Rightarrow \;\; &
\begin{cases} 
 (\al_1p-\al_2)r+(\al_1p-\al_2)\varrho &= \ \al_1p+\al_2 \\
-(\al_2p+\al_1)r+(\al_2p+\al_1)\varrho &= \ \al_1-\al_2p \\
 (\bt_1p-\bt_2)r+(\bt_1p-\bt_2)\varrho &= \ \bt_1p+\bt_2 \\
-(\bt_2p+\bt_1)r+(\bt_2p+\bt_1)\varrho &= \ \bt_1-\bt_2p
\end{cases}\label{E:sysR&Rho} \\
& \quad \text{\em has a solution in} \; \RR\setminus((\{0\}\times\rea)\cup \{(1,0)\}) \notag \\
& \quad \text{\em for some $(A,B)\in \CC$ such that $A\ba{B}\notin \rea$.} \notag
\end{align}
Let us restrict ourselves to $p=q \neq 0$. In this case, if $(\al_1p-\al_2)=0,$ then the consistency of the
above system of equations forces on us:
\[
 (\al_1p-\al_2)=0 ~~\text{and}~~(\al_1p+\al_2)=0.
\]
That implies $\al_1+i\al_2=0$, which contradicts the invertibility of $T$. Thus $\al_1p-\al_2 \neq 0$.
Similarly, all the
coefficients of the left hand side of the above system of equations \eqref{E:sysR&Rho} are non-zero. Thus,
$T$ having the mapping properties given in \eqref{E:order1} exists
\begin{align}
\Rightarrow \;\; &
\begin{cases}
 \dfrac{\al_1p+\al_2}{\al_1p-\al_2} &= \ \dfrac{\bt_1q+\bt_2}{\bt_1q-\bt_2}, \\
 {} & {} \\
 \dfrac{\al_1-\al_2p}{\al_1+\al_2p} &= \ \dfrac{\bt_1-\bt_2q}{\bt_1+\bt_2q}
\end{cases} \notag\\
\Rightarrow \;\; &
\begin{cases}
 \dfrac{\al_2}{\al_1p-\al_2} &= \ \dfrac{\bt_2}{\bt_1q-\bt_2}, \\
 {} & {} \\
 \dfrac{\al_1}{\al_1+\al_2p} &= \ \dfrac{\bt_1}{\bt_1+\bt_2q}
\end{cases} \notag \\
\Rightarrow \;\; &
\begin{cases}
 \al_1\bt_2-\bt_1\al_2 & \ =0, \\
 \al_2\bt_1-\bt_2\al_1 & \ =0 \;\; (\text{since}\; p=q \neq 0).
\end{cases} \notag
\end{align}
But the second condition implies that $\imag(A \ba{B})=0$, i.e. $A\ba{B} \in \rea$, which is a contradiction.
Thus there is no $T$ with the mapping properties given in \eqref{E:order1}.
\smallskip

Under the assumption $p=q$, we still need to show that there is no invertible $\cplx$-linear map
that maps $P_1\cup P_2$ to the union of the two graphs given in \eqref{E:order1}, but with the images swapped.
Ruling this out is a shorter argument. In this case, $T$ will have the matrix representation given by
\eqref{E:mat2}. Hence, $T$ having the desired properties exists (this time, we implicitly view the necessary
conditions as a linear system with $\al_1$, $\al_2$, $\bt_1$, $\bt_2$ as unknowns)
\begin{align}
 \Rightarrow \;\; &
 \begin{cases}
        p(r+\varrho-1)\al_1+(r-\varrho+1)\al_2 &= \ 0 \\
        (r+\varrho+1)\al_1+p(\varrho-r+1)\al_2 &= \ 0 \\
        p(r+\varrho-1)\bt_1+(r-\varrho+1)\bt_2 &= \ 0 \\
        (r+\varrho+1)\bt_1+p(\varrho-r+1)\bt_2 &= \ 0
        \end{cases} \notag \\
 & \quad \text{\em has a solution in $\rea^4$ such that $\al_1\bt_2-\bt_1\al_2\neq 0$} \notag \\
 & \quad \text{\em for some $(r,\varrho)\in \RR\setminus ((\{0\}\times\rea)\cup \{(1,0)\})$.} \notag
\end{align}
Hence, every $(\al_1,\al_2)$ and  $(\bt_1,\bt_2)$ such that $(\al_1,\al_2,\bt_1,\bt_2)$ is a solution
of the above system will be solutions in $\RR$ of the
following system of equations
\begin{align}
 p(r+\varrho-1)X+(r-\varrho+1)Y & =0 \notag\\
(r+\varrho+1)X+p(\varrho-r+1)Y & =0 .\notag
\end{align}
For the matrix $M_T$ in \eqref{E:mat2} to be invertible, we need $\{(\al_1,\al_2),(\bt_1,\bt_2)\}$ to be linearly
independent in $\RR$. The only way we can get $\{(\al_1,\al_2),(\bt_1,\bt_2)\}$ to be linearly independent is for each
coefficient of the above system of equations to vanish. This gives $r=0$, which is a contradiction. Hence a
$T$ with the matrix representation \eqref{E:mat2} having the other desired properties cannot exist.
\smallskip

It follows from the work of Weinstock \cite{Wk} (refer to the last paragraph of Section~\ref{S:technical} for a precise
statement) that, by our choice of $(s,t)$ and $(\sm,\tau)$, $P_1\cup P_2$ is locally polynomially convex.
To conclude: it can easily be checked that the transformations $S_p$ determined (with respect to the standard
basis) by the matrices
\[
 \begin{pmatrix}
   p+i & 0 \\
   0 & p+i
  \end{pmatrix}, \;\; p\in \rea\setminus\{0\},
\]
give us the 1-parameter family $\{S_p:p\in \rea\setminus\{0\}\}$ having all the properties
stated in Observation~\ref{O:counterEx}.
\medskip

\subsection{A discussion on the correctness of Result~\ref{R:dieuWrong}}\label{SS:counterDieu}
Another issue that --- as we discussed in Section \ref{S:intro} --- needs to be settled is the status of 
Result \ref{R:dieuWrong}. We address this now. First, we shall show the following:
There exist $\lm \in \rea \setminus \{0,-1\}$ and a real-valued function $\varphi \in \smoo^1(\{0\})$ such that,
 $S_1\cup S_2$ in \ref{R:dieuWrong} {\em is not polynomially convex at} $(0,0)$,  and yet
$\{z \in \cplx: \rl{z}=t/2,~~2 \lm \rl{z}+\varphi(z)= 0\}$ {\em has more than one connected components} for all
$t>0$. The proof goes as follows:
\smallskip

Let $\varphi(z)= (\imag{z})^2= y^2$ (writing $z=x+iy$), and consider the polynomial $p(z,w)=z+w.$ The polynomial $p$ is
 real valued when it is restricted to $S_1\cup S_2.$ Let $V_t= p^{-1}\{t\}$. Now let us compute $V_t \cap S_j,~~j=1,2.$
 We have
\begin{align}
V_t\cap S_1&= \{(z,\zbar ):  \rl{z}=t/2\},\notag\\
V_t\cap S_2&= \{(z,\zbar+\ba{\lm}z+\lm \zbar ): 2 \rl{z}+2 \lm \rl{z}+(\imag{z})^2=t\}.\notag
\end{align}
Let $\pi_1$ denote the projection onto the first coordinate. Then, the above are curves in $\CC$ that project 
down to:
\begin{align}
 \pi_1(V_t\cap S_1)&= \{(x,y)\in \RR~:~ x=t/2\},\notag\\
\pi_1(V_t\cap S_2)&= \{(x,y)\in \RR~:~ 2x+2 \lm x+y^2=t\}. \notag\\
&= \left\lbrace(x,y)\RR~:~ x-\frac{t}{2(1+\lm)} = - \frac{y^2}{2(1+\lm)}\right\rbrace. \notag
\end{align}
Let us now choose $\lm~:~-1<\lm<0,$ and fix it. For $t>0$, we see that:
\begin{itemize}
 \item [(1)] $\pi_1(V_t\cap S_1) \cap \pi_1(V_t\cap S_2)$ consists of the two points $(t/2, \pm \sqrt{t|\lm|})$; and
\item[(2)] $\cplx \setminus \pi_1(V_t\cap S_1) \cup \pi_1(V_t\cap S_2)$ contains a bounded component, 
say $\mathscr{D}_t,$ and $\mathscr{D}_t \rightarrow \{0\}$ as $0<t\searrow 0.$
\end{itemize}

Let us now write $\pi_1(V_t\cap S_1) \cap \pi_1(V_t\cap S_2)= \{\zt_1(t),\zt_2(t)\},~ t>0.$ Note:
\begin{equation} \nonumber
 \pi_1^{-1}\{\zt_j(t)\}\cap S_1=\{(\zt_j(t),t-\zt_j(t))\}= \pi_1^{-1}\{\zt_j(t)\}\cap S_2,~~j=1,2,~ t>0.
\end{equation}
From this and (2), we conclude that $V_t\cap(M_1\cup M_2)$ determines a closed curve $C_t$ such
that
\begin{equation}\nonumber
 \pi_1(C_t)=\bdy{\mathscr{D}_t},~~\forall t~:~0<t\searrow 0.
\end{equation}
Hence, we get a family of analytic discs $\psi_t: \mathscr{D}_t \longrightarrow \CC$ by $z\longmapsto (z,t-z),$
attached to $S_1\cup S_2$ and $\psi_t \rightarrow 0$ as $t\searrow 0.$ Hence, by maximum modulus theorem, 
$S_1\cup S_2$ is not locally polynomially convex at the origin. Yet, owing to (1),
 $$ \left\lbrace z \in \cplx: \rl{z}=\frac{t}{2},~~2 \lm \rl{z}+(\imag(z))^2= 0\right\rbrace$$ has two connected 
components for all $t>0.$
\smallskip

This shows that condition $[(i)~~ \text{AND}~~ (ii)]$ in Result \ref{R:dieuWrong} is not always necessary
 for $S_1 \cup S_2$ to not be locally polynomially convex at $(0,0) \in \CC$. 
\smallskip

Now we shall show that there exists $\lm \in \rea \setminus \{0,-1\}$ and a real-valued function 
$\varphi \in \smoo^1(\{0\})$ such that the condition (ii) is satisfied and yet 
$S_1 \cup S_2$ in  Result \ref{R:dieuWrong} is locally polynomially convex at $(0,0) \in \CC$. 
Let us consider the following surfaces in $\CC$:
\begin{align}
 S_1 &:= \{(z,w)\in \CC: w=\zbar \}, \notag \\
 S_2 &:= \{(z,w)\in D(0;\dl)\times \cplx: w=(1+\lambda)\zbar + \ba{\lambda}z + \varphi(z)\}, \notag
\end{align}
 where $\lm \in \rea \setminus \{0,-1\}$, $\dl>0$ and
\[
\varphi = \left. \Phi(\rl(\bcdot))\right|_{D(0;\dl)}, 
\]
where $\Phi \in \rea[x]$, i.e. a polynomial in $x:=\rl{z}$ with 
real coefficients, such that $\Phi(0)=0=\Phi'(0)$. Let us consider the polynomial $P(z,w)=z+w$. 
Now let us compute the set $P^{-1}\{t\}\cap S_j$ for $j=1,2$.
\begin{align}
 P^{-1}\{t\}\cap S_1 &= \{(z,\zbar)\in \CC : z+\zbar=t\} \notag\\
 &=\{(t/2,+iy,t/2-iy)\in \CC : y\in \rea\}\;\;\;(\text{writing}\; z=x+iy), \label{E:pinvintS1}
\end{align}
\begin{align}
 P^{-1}\{t\}\cap S_2 &= \{(z,w)\in D(0;\dl)\times \cplx : w=(1+\lambda)\zbar + \ba{\lambda}z 
+ \varphi(z),\; z+w=t\} \notag\\
&= \{(x+iy,x-iy+2\lm x+ \varphi(x))\in D(0;\dl)\times \cplx : 2(1+\lm)x+\Phi(x)=t\}. \label{E:pinvintS2}
\end{align}
Let $q_t(x)= 2(1+\lm)x+\Phi(x)-t$ and let $Z_\rea(q_t)$ denote the set of real zeros of the polynomial $q_t$. 
We have:
\begin{align}
& \left\lbrace z \in D(0;\dl): \rl{z}=\frac{t}{2},~~2 \lm \rl{z}+\Phi(z)= 0\right\rbrace \notag\\
&= \begin{cases}
    \varnothing,\;\;\text{if $t/2 \notin Z_\rea(q_t)$}, \\
    \{(t/2+iy): y\in \rea\}\cap D(0;\dl),& \text{if $t/2 \in Z_\rea(q_t)$}.
   \end{cases} \notag
\end{align}
This shows that if we fix $\dl>0$ to be sufficiently small then, because $Z_\rea(q_t)\cap (-\dl,\dl)$ is
at most a singleton, the set 
$\left\lbrace z \in \cplx: \rl{z}=\frac{t}{2},~~2 \lm \rl{z}+\varphi(z)= 0\right\rbrace$
has at most one component, for all $t \in \rea$ sufficiently small. Hence the condition (ii) 
of Result \ref{R:dieuWrong} holds.
\smallskip

From \eqref{E:pinvintS1} and \eqref{E:pinvintS2}, we have the following:
\begin{multline}
P^{-1}\{t\} \cap ((S_1 \cup S_2)\cap \ba{D(0;\eps)} \times \cplx) \\
= \{(t/2,+iy,t/2-iy)\in \ba{D(0;\eps)} \times \cplx : y\in \rea \}\\
 \cup \left( \cup_{x \in Z_\rea(q_t)\cap[\eps,\eps]}
 \{ (x+iy,-x-iy+t) : \sqrt{x^2+y^2}\leq \eps \} \right). \notag 
\end{multline}
These are lines segments in $\ba{D(0;\eps)} \times \cplx$ whose projections on $\cplx_z$ are line segments parallel 
to $y$-axis. Hence, $P^{-1}\{t\} \cap ((S_1 \cup S_2)\cap \ba{D(0;\eps)} \times \cplx)$ is union of finitely many 
non-intersecting line segments when $\eps \in (0,\dl)$. Hence 
$$(P^{-1}\{t\} \cap ((S_1 \cup S_2)\cap \ba{D(0;\eps)} \times \cplx )\hull{)}= P^{-1}\{t\} \cap ((S_1 \cup S_2)\cap 
\ba{D(0;\eps)} \times \cplx),$$
 for $\eps \in (0,\dl)$ and for all $t \in \rea$ sufficiently small. Hence, the pair $(S_1, S_2)$ 
satisfies the conditions (i) and (ii) in Result \ref{R:dieuWrong} and yet $S_1\cup S_2$ is locally 
polynomially convex at the origin in $\CC$. This last assertion follows from a very useful result ---
see Result \ref{R:S1} below--- for computing polynomial hulls.
\smallskip
 
This shows that condition $[(i)~~\text{ AND}~~ (ii)]$ in Result \ref{R:dieuWrong} is not sufficient for $S_1 \cup S_2$
to not be locally polynomially convex at $(0,0) \in \CC$.
\medskip

\section{Technical preliminaries}\label{S:technical}
We shall require a couple of preliminaries to set the stage for proving the above theorems. The principal
tool that we shall use is the following lemma by Kallin \cite{K}.
\begin{lemma}[Kallin]\label{L:kallin}
 Let $K$ and $L$ be two compact polynomially convex subsets in $\cplx ^n$. Suppose
 there exists a holomorphic polynomial $P$ satisfying the following
 conditions:
 \begin{enumerate}
 \item[$(i)$] $\widehat{P(K)} \cap \widehat{P(L)}\subseteq \{0\}$ and $0 \in \cplx \setminus int(\widehat{P(K)}
  \cup \widehat{P(L)})$; and
 \item[$(ii)$] $P^{-1}\{0\} \cap (K \cup L)$ is polynomially convex.
 \end{enumerate}
 Then $K\cup L$ is polynomially convex.
\end{lemma}
The other tool we shall use in the course of the proof of some of the above theorems is the following theorem
from Stout's book \cite[Theorem 1.2.16]{S1}.
\begin{result}\label{R:S1}
 If $X\subset \cplx^n$ is compact and if $\mathscr{P}(X)$ contains a real valued function, say f, then $X$ is 
polynomially convex if and only if each fiber $f^{-1}\{t\}\cap X$, $t \in \rea$, is polynomially convex. If 
$X$ is polynomially convex, then $\mathscr{P}(X)=\smoo(X)$ if and only if for each $t$, 
$\mathscr{P}(f^{-1}\{t\}\cap X)=\smoo(f^{-1}\{t\}\cap X)$.
\end{result}
\noindent Here, $\mathscr{P}(X)$ denotes the uniform algebra generated by all holomorphic polynomials restricted to $X$.

Let $S_{1}$ and $S_{2}$ be two totally-real surfaces in $\cplx^2$ passing through the origin.  Their
tangent spaces at the origin are also totally real. If $T_0 S_1 \cap T_0 S_2=\{0\}$, then there exist
global holomorphic coordinates $(z,w)$ with respect to which  $T_{0}S_{1}= \mathbb{R}^2$ and
$T_{0}S_{2}\,=\, M(A)$ for some real matrix $A$, where $A+iI$ is invertible and $M(A):=\,(A+iI)\mathbb{R}^2$.
Here $\RR:= \{(z,w):~ \imag{(z)}=0=\imag{(w)}\}$. The reader is referred to Weinstock's paper \cite{Wk} for details.
\smallskip

\indent Near the origin, $S_{1}$ and $S_{2}$ will be small perturbations of $\mathbb{R}^2$ and $M(A)$ respectively.
Define $S_j(\dl):= S_j \cap \overline{B(0;\dl)},~~~j=1,2$. For sufficiently small $\dl >0$, we have
\begin{align}
S_{1}(\delta)&= \{ (x+ f_{1}(x,y), y+ f_{2}(x,y)) : x,y \in \mathbb{R} \} \cap \overline{B(0; \delta)} \notag \\
 S_{2}(\delta)&= \{(A+iI)(x,y)+(g_{1}(x,y), g_{2}(x,y)) : x,y \in \mathbb{R} \} \cap \overline{B(0;\delta)}, \notag
\end{align}
where $f_j,g_j = o(||(x,y)||)$ as $(x,y)\rightarrow 0$ are $\cplx$-valued functions, $j=1,2$.
\smallskip

 Since $T_{0}S_{1}\cup T_{0}S_{2}$ is locally polynomially convex at origin, it satisfies Weinstock's criterion \cite{Wk},
 i.e. $A$ has no purely imaginary eigenvalue of modulus greater than 1. It is easy to show that the image of $M(A)\cup \RR$
 under a $\cplx$-linear transformation represented by a {\em real} nonsingular matrix $S$ is $M(SAS^{-1}) \cup \RR$, whence
 this transformation maps $S_{1}(\dl)\cup S_{2}(\dl)$ to $\widetilde{S_{1}}(\dl)\cup \widetilde{S_{2}}(\dl)$, where
\begin{align}
\widetilde{ S_1}(\delta)&=\{ (x+  \widetilde{f_{1}}(x,y),y+ \widetilde{f_{2}}(x,y)): x,y \in \mathbb{R} \}
\cap S(\ba{B(0;\delta)})\notag \\
\widetilde{S_{2}}(\delta)&=\{(SAS^{-1}+iI)(x,y)+(\widetilde{g_{1}}(x,y),\widetilde{g_{2}}(x,y)) : x,y \in \mathbb{R}\}
\cap S(\ba{B(0;\delta)}), \notag
\end{align}
where $\wtil{f_j}, \wtil{g_j}$ have the same properties as $f_j, g_j$ given above, $j=1,2$.
\medskip

\section{The proof of Theorem~\ref{T:transversePCVX}}\label{S:proof-transversePCVX}
Let $S_{1}$ and $S_{2}$ be two totally-real surfaces intersecting only at the origin and $T_0S_1 \cap 
T_0S_2=\{0\}$.
Let
\begin{align}
S_{1}(\delta)&= \{ (x+ f_{1}(x,y), y+ f_{2}(x,y)) : x,y \in \mathbb{R} \} \cap \overline{B(0; \delta)}\notag \\
S_{2}(\delta)&= \{(A+iI)(x,y)+(g_{1}(x,y), g_{2}(x,y)) : x,y \in \mathbb{R} \} \cap \overline{B(0;\delta)}, \notag
\end{align}
where $f_j, g_j$ are as described in Section \ref{S:technical}. Now, it is a fact of basic linear algebra that every real
$ 2 \times 2$ matrix is similar via a {\em real} nonsingular matrix, to one of the following three kinds of matrices:
a diagonal matrix with real entries, a matrix of the form 
$ \begin{pmatrix}
 \lambda & 1 \\
 0 & \lambda
\end{pmatrix} $ or of the form 
$\begin{pmatrix}
 s &-t\\
 t & s
 \end{pmatrix}$ 
 where $\lambda, s,t \in \mathbb{R}$. Given this fact, and the argument in the last paragraph of of Section
 \ref{S:technical}, the proof of Theorem \ref{T:transversePCVX} reduces to the following two lemmas.
This is because it is sufficient to take the matrix $A$ to be one of the above form.
\begin{lemma}\label{L:realEig}
If $A\,=\,\begin{pmatrix}
 \lambda & 1 \\
 0 & \lambda
\end{pmatrix}$,
where $\lambda \in \mathbb{R}$ or $A$ is a diagonal matrix with real entries, then $S_{1} \cup S_{2}$ is locally 
polynomially convex at origin.
\end{lemma}
\begin{proof}
We shall show that, shrinking $\delta$ if necessary, $S_{1}(\delta)\cup S_{2}(\delta)$
is polynomially convex. Consider the polynomial \[ P(z)\,=\,\langle(A-iI)z, z\rangle \]
where $\langle z,w \rangle :=\, z_{1}w_{1}+ z_{2}w_{2}$. We will first consider the case when $A$ is a Jordan block. 
\begin{align}
 P(x+ f_{1}(x,y), y+ f_{2}(x,y))\,
 &= \langle((\lambda-i)x+y, (\lambda-i)y), (x,y)\rangle + H( x,y)\notag \\
 &= ((\lambda-i)x+y)x + (\lambda-i)y^2 + H(  x,y) \notag \\
 &= (\lambda-i)x^2 + xy + (\lambda -i) y^2 +  H(  x,y) \notag \\
 &= (\lambda x^2 + xy + \lambda y^2) - i(x^2 +y^2) + H( x,y), \notag
\end{align}
  where $H(x,y)= o(||(x,y)||^2)$ as $(x,y)\rightarrow 0$.
  \smallskip

Since $\lim_{(x,y)\rightarrow 0} H(x,y)/||(x,y)||^2 =0$, taking $\dl>0$ sufficiently small,  
\[ \imag(P(z)) < 0 \,\,\, \forall z \in S_{1}(\delta) \setminus \{0\}\] and equal to zero only when $z=0$.
\smallskip

Now, for $z \in S_2(\dl)$
\begin{align}
P(z) &= P((\lambda + i)x+y+ g_{1}(x,y), (\lambda +i)y +g_{2}(x,y)) \notag \\
&= \langle((\lambda ^2 +1)x +2 \lambda y , (\lambda ^2 + 1)y), ((\lambda +i)x + y , (\lambda +i)y)\rangle + o(||(x,y)||^2) \notag \\
&= [(\lambda +i)x + y][(\lambda ^2 + 1)x + 2 \lambda y] + (\lambda ^2 + 1)(\lambda +i) y^2 + o(||(x,y)||^2) \notag \\
&= [(\lambda ^2 + 1)\lambda x^2 + (2 \lambda ^2  + (\lambda ^2 + 1)) xy + (2 \lambda + \lambda ^3)y^2] \notag \\
&\quad + i[(\lambda ^2 + 1)x^2 + 2 \lambda xy + (\lambda ^2 + 1)y^2 ]+o(||(x,y)||^2). \notag
\end{align}
 Here
\begin{align}
  \imag(P(z)) &= (\lambda ^2 + 1)x^2 + 2 \lambda xy + (\lambda ^2 + 1)y^2 +o(||(x,y)||^2)\notag \\
  &= \lambda ^2 x^2 + y^2 + (x+ \lambda y)^2 + o(||(x,y)||^2).\notag
\end{align}
  So, shrinking $\dl>0$ if necessary,
 \[ 
\imag(P(z)) > 0 \,\,\,\, \forall z \in S_{2}(\delta) \setminus \{0\}
 \]
 and equal to zero only when $z=0$.
\smallskip

We can now show that $P^{-1}\{0\}\cap (S_{1}(\delta) \cup S_{2}(\delta))$ is polynomially convex.
Observe that $P^{-1}\{0\}\cap (S_{1}(\delta) \cup S_{2}(\delta))\,=\, \{(0,0)\}$, hence polynomially convex.
 By Lemma \ref{L:kallin} $S_{1}(\delta)\cup S_{2}(\delta)$ is polynomially convex.
 The same proof goes through with $P$ defined exactly as above (and with considerably simpler calculations) when 
$A$ is a diagonal matrix with real entries.
\end{proof}
We now have to consider the case when $A\,=\,
 \begin{pmatrix}
 s &-t\\
 t & s
 \end{pmatrix}$. Recall that, by hypothesis, $M(A)\cup \RR$ is locally polynomially convex at $0\in \CC$.
By Weinstock's criterion \cite{Wk}, $t\in \rea$ will satisfy $|t|<1$ whenever $s=0$. It is this requirement that shapes 
our next lemma.

\begin{lemma}\label{L:cplxEig}
 If $A\,=\,
 \begin{pmatrix}
 s &-t\\
 t & s
 \end{pmatrix}$ and $|t|<1$ whenever $s\,=\,0$, then $S_{1}\cup S_{2}$ is locally polynomially convex at origin.
\end{lemma}
\begin{proof}
  As before
\begin{align}
  S_{1}(\delta)&= \{ (x+ f_{1}(x,y),y+f_{2}(x,y)): ||(x,y)||\leq \delta\} \notag\\
  S_{2}(\delta)&= \{( (s+i)x-ty+ g_{1}(x,y),tx+(s+i)y+g_{2}(x,y)): ||(x,y)||\leq \delta \}. \notag
 \end{align}
Consider the polynomial
  \[F(z_{1},z_{2})\,=\,z_{1}^{2} + z_{2}^{2}.\]
 So,
 \[F(x+ f_{1}(x,y),y+f_{2}(x,y)\,=\, x^2 +y^2 + H_1(x,y),\]
 where $H_1(x,y)= o(||(x,y)||^2)$ as $ (x,y) \rightarrow 0$.
\begin{multline}
 F((s+i)x-ty+ g_{1}(x,y),tx+(s+i)y+g_{2}(x,y))\\
=(s^2+t^2-1)x^2+ (s^2+t^2-1)y^2 +2si(x^2+y^2)+ H_{2}(x,y), \notag
\end{multline}
  where $H_{2}(x,y)= o(||(x,y)||^2)$ as $(x,y)\rightarrow 0$.
  \smallskip

 Hence, for $\dl>0$ sufficiently small, 
\[\rl(F(z))\geq 0 \,\,\,\forall z \in S_{1}(\delta)\] and equal to zero only when $z=0$.
 \medskip

 \noindent{\bf {\em Case I}.} {\em When $(s^2+t^2)<1$}.
 \smallskip

 \noindent{ Clearly, after shrinking $\dl>0$ if necessary, $\rl(F(z))\leq 0 \; \forall \; z\in S_{2}(\delta)$ 
 and equal to zero only when $z=0$.}
 \medskip

 \noindent{\bf {\em Case II}.} {\em When $s^2+t^2 \geq 1$}.
 \smallskip

 \noindent First note that, by hypothesis, $s \neq 0$ in this case. 
We fix an $\eps>0$ sufficiently small, whose precise value will be specified later. Then,
 since $\lim_{(x,y)\rightarrow 0} H_2(x,y)/||(x,y)||^2 =0$, $\exists \dl_\eps >0$ such that
 \[ 
F(S_{2}(\dl_\eps)) \subset \{ u+iv \in \cplx : |(s^2+t^2-1)v- 2su|< \eps |v|\} .
\]
 Call the set in the right hand side as $\mathfrak{C}_{2, \eps}.$
 In fact shrinking $\dl_\eps$ further if necessary we shall also get
 \[F(S_{1}(\dl_\eps) \subset \mathfrak{C}_{1, \eps} := \{u+iv \in \cplx : |v |< \eps u \}. \]
 Now choose sufficiently small $\eps_0>0$ such that
 \[\mathfrak{C}_{1, \eps_0} \cap \mathfrak{C}_{2, \eps_0}\,=\, \{0\},\]
and write $\dl=\dl_{\eps_0}$.
\smallskip

Hence, in both the cases $F(S_{1}(\dl))$ and $F(S_{2}(\dl))$ lie in two different angular sectors
intersecting only at the origin. We also have 
\[ 
F^{-1}\{0\}\cap (S_{1}(\dl) \cup S_{2}(\dl))=\{(0,0)\}
\] 
is polynomially convex. So, all the hypotheses of Lemma~\ref{L:kallin} are satisfied.
Hence $S_{1}(\dl) \cup S_{2}(\dl)$ is polynomially convex.
\end{proof}

In view of our remarks above, Lemmas~\ref{L:realEig} and \ref{L:cplxEig} give us the result. \qed
\medskip 

\section{The proof of Theorem~\ref{T:epsperturb}}\label{S:proof-epsperturb}
Let $P$ be a $\cplx$-linear function such that
\[ \upl \subset \mathbb{H}:=\{(z,w)\in \CC :~ \imag{P(z,w)}=0 \} .\]
By interchanging the roles of $z$ and $w$ if necessary, we may assume that $\bdy_wP \not \equiv 0$.
Now consider the biholomorphic map
$ \Phi : \begin{pmatrix}
 z \\
 w
\end{pmatrix} \mapsto
\begin{pmatrix}
 z\\
 P(z,w)
 \end{pmatrix}$
 from $\CC$ to $\CC$.
 \smallskip

 We have \[\Phi(\upl) \subset \cplx_z \times \rea_u \qquad\qquad \left(\text{taking}~ w=u+iv \right). \]
 Since $span_\cplx \{M_1 \cap M_2\} \subset \mathbb{H}$ and $\mathbb{H}$ contains a unique complex line namely
 $\{(z,w)\in \CC :~ P(z,w)=0\}$, $ \Phi(span_\cplx \{M_1 \cap M_2\})= \cplx_z \times \{0\}$ and hence
 $ M_1 \cap M_2 \subset \cplx_z \times \{0\}.$
 Now we can find a $\tht \in [0,2\pi)$ such that if $\Psi:=  (e^{i\tht}\Phi_1,\Phi_2)$, then
 \begin{align}
 \Psi(M_1 \cap M_2) &= \{(x,0)\in \CC :~ x \in \rea \}, \;\; \text{and} \notag \\
 \Psi(\upl) &\subset \cplx_z \times \rea_u . \notag
 \end{align}
 \smallskip

Let us find the equations of $\Psi(M_j), ~j=1,2$. The analysis reduces to exactly two cases.
 \medskip

 \noindent{{\bf {\em Case I}.} {\em When neither $\Psi(M_1)$ nor $\Psi(M_2)$ is perpendicular to $\cplx_z \times \{0\} $}.
 \smallskip

 \noindent{In this case $\Psi(M_1)$ and $\Psi(M_2)$ both can be written in the graph form. Writing $z=x+iy$, 
  we get:}
 \smallskip

 Equation of $\Psi(M_j)=
 \begin{cases}
 A_j x+B_jy+D_ju &=0, \\
\qquad\qquad\qquad\quad v&=0, ~j=1,2.
 \end{cases}
 $
 \smallskip

 Both the planes $\Psi(M_j), ~j=1,2$, pass through $\{(x,0)\in \CC : x\in \rea \}$. Hence, 
 $A_j=0,~ j=1,2$, and hence, there exist $C_1, C_2 \in \rea \setminus \{0\}, ~ C_1 \neq C_2$ such that
 \begin{align}
 \Psi(M_1)&=\{(x+iy, C_1y)\in \CC :~ x+iy \in \cplx\}, \notag \\
 \Psi(M_2)&=\{(x+iy, C_2y)\in \CC :~ x+iy \in \cplx\} .\notag
 \end{align}
 Now, for $\eps >0$, write
 \[ F_j^\eps := \eps x^2+ \phi_j(z),\]
 where $\phi_j$ are real valued functions with $\phi_j(z)=o(|z|^2), $ and set
 \[ \wtil{S_j^\eps}:= \{(x+iy, C_jy+ F_j^\eps(x,y)) :~ x,y \in \rea\} ,~~j=1,2 .\]
 Consider the two parabolas in $\cplx_z$: $\mathfrak{Q}_j(\eps,\dl):=
 \{x+iy \in \cplx :~ (y-\dl/C_j)= -(\eps/C_j)x^2 \}$, $j=1,2$,
 and the following small perturbations of the above parabolas
 \[\wtil{\mathfrak{Q}}_j(\eps,\dl):= \{ x+iy \in \cplx :~ C_jy+ F_j^\eps(x,y)=\dl\},~j=1,2,\]
 where $\dl>0$ is sufficiently small.
 \smallskip

 It is an absolutely elementary fact that $\cplx_z \setminus (\mathfrak{Q}_1(\eps,\dl)\cup \mathfrak{Q}_2(\eps,\dl))$
 has a bounded component $\mathfrak{D}(\eps,\dl)$ and hence, for each $\eps>0$, there exists $\Delta_0(\eps)>0$ such that
 $\cplx_z \setminus (\wtil{\mathfrak{Q}}_1(\eps,\dl)\cup \wtil{\mathfrak{Q}}_2(\eps,\dl))$ has a bounded component
 $\wtil{\mathfrak{D}}(\eps,\dl)$ for all $\dl \in (0, \Delta_0(\eps)). $
 \smallskip

 Hence, $\mathfrak{A}_\dl:=\ba{\wtil{\mathfrak{D}}(\eps,\dl)} \times \{\dl\}$ are closed analytic discs with
 boundaries in $\wtil{S_1^\eps}\cup \wtil{S_2^\eps}$ for each $\dl \in (0, \Delta_0(\eps))$ and
 $\mathfrak{A}_\dl \rightarrow \{0\}$ as $\dl\searrow 0$.
 \smallskip

 Clearly $\wtil{S_1^\eps}\cup \wtil{S_2^\eps}$ is not polynomially convex at the origin.
 Hence, $S_j^\eps := \Psi^{-1}(\wtil{S_j^\eps}), ~j=1,2,$ are the required perturbations.
 \medskip

 \noindent{{\bf {\em Case II}.} {\em When one of $\Psi(M_j),~ j=1,2$, is perpendicular to $\cplx_z \times \{0\}$.}
 \smallskip

 \noindent{Let us assume that $\Psi(M_1)$ is perpendicular to $\cplx_z \times \{0\}$. So, $\Psi(M_2)$ can be 
written in graph form.}
 \smallskip

 We have,
 \begin{align}
 \Psi(M_1)&=\{(x,u)\in \CC :~  x,u \in \rea\} \notag \\
 \Psi(M_2)&=\{(x+iy, Cy) \in \CC :~x,y \in \rea \} \notag
 \end{align}
 where $C\in \rea \setminus \{0\}$. Choosing $F_2^\eps$ exactly same as in {\em Case I} we write:
 \begin{align}
 \wtil{S_1^\eps} &:= M_1 \notag \\
 \wtil{S_2^\eps}&:=\{(x+iy, Cy +F_2^\eps(x,y))\in \CC :~ x,y \in \rea \} . \notag
 \end{align}
 An analysis entirely similar to {\em Case I} will yield a $\Delta_0(\eps)>0$ and closed analytic discs $\mathfrak{A}_\dl$
 with boundaries in $\wtil{S_1^\eps}\cup\wtil{S_2^\eps}$ such that $\mathfrak{A}_\dl \rightarrow \{0\}$ as $\dl\searrow 0$.
 \smallskip

 As before, $S_j^\eps:= \Psi^{-1}(\wtil{S_j^\eps}), ~j=1,2,$ are the required perturbations. \qed
 \medskip

\section{The proof of Theorem~\ref{T:linePcvxFlat}}\label{S:proof-linePcvxFlat}
Since $dim_\rea (T_0S_1 \cap T_0S_2)=1$, $span_\rea(T_0S_1 \cup T_0S_2)$ is a real three dimensional subspace
of $\CC$ and there exists a $\cplx$-linear map $P: \CC \mapp \cplx$ such that
\[
 span_\rea(T_0S_1 \cup T_0S_2)= \left\lbrace(z,w)\in \CC~:~ \imag(P((z,w))=0\right\rbrace= :\mathbb{H}.
\]
The condition $S_1\cup S_2 \subset span_\rea(T_0S_1 \cup T_0S_2)$ implies $S_1\cup S_2 \subset \mathbb{H}$. 
Therefore we have:
\[
 (S_1(\dl)\cup S_2(\dl)\hull{){\;\;}} \subset cvx(S_1(\dl)\cup S_2(\dl)) \subset \mathbb{H},
\]
where $S_j(\dl)=S_j\cap \ba{B(0;\dl)},~~j=1,2$ (here, $B(0;\dl)$ denotes a ball in $\CC$ centred at origin
 and having radius $\dl$) and $cvx(S)$ denotes the convex hull of $S$. We consider the biholomorphism
 $ \Phi : \begin{pmatrix}
 z \\
 w
\end{pmatrix} \mapsto
\begin{pmatrix}
 z\\
 P(z,w)
 \end{pmatrix}$
(as before, we may assume, interchanging the roles of $z$ and $w$ if necessary, that $\bdy_wP \not \equiv 0$
 from $\CC$ to $\CC$ which has the following effect:
\begin{align}
\Phi(S_1\cap S_2)&\subset \cplx_z \times \rea_u, \;\;(\text{where} \; z=x+iy,\; w=u+iv) \notag\\
[\Phi(S_1(\dl)\cap S_2(\dl))\hull{]{\;\;}} &\subset \cplx_z \times \rea_u .\notag
\end{align}
\smallskip

Our examination involves exactly two cases. Let $\pi_z$ denote the projection onto the first coordinate.
\smallskip

\noindent {{\bf {\em Case I}.} {\em When $\pi_z[\Phi(T_0S_1\cap T_0S_2)]$ is a line in  $\cplx_z \times \{0\}$}.
\smallskip

\noindent We make one final adjustment. Let $\tht$ be the angle between the line $\{(x,0):\; x  \in \rea\}$  
and $\pi_z[\Phi(T_0S_1\cap T_0S_2)]$ in $\cplx_z \times \{0\}$, and let $\Psi:=(e^{-i\tht}\Phi_1, \Phi_2)$. Note that, from
the assumption $span_\cplx\{T_0S_1\cap T_0S_2\}\nsubseteq \ span_{\rea}\{\utgt{S_1}{S_2}\}$, we get 
$\Psi(T_0S_1 \cap T_0S_2)$ is not the $x$-axis. Hence $\exists a\in \rea \setminus \{0\}$ such that
\[
 \Psi(T_0S_1 \cap T_0S_2)\;\;: y=0,\; v=0,\; u=ax.
\]
Furthermore, we have:
\smallskip

Equation of $\Psi(T_0S_j)=
 \begin{cases}
 u&= ax+B_jy, \\
 v&=0,  \; B_j \in \rea, ~j=1,2, \; \text{and}\; B_1 \neq B_2.
 \end{cases} $
\smallskip

For sufficiently small $\dl>0$,
\begin{align}
 \wtil{S_j}(\dl):= \Psi(S_j)\cap \ba{B(0,\dl)}& = 
\begin{cases}
 u&= ax+B_jy+\varphi_j(x,y), \\
 v&=0,
\end{cases} \notag
\end{align}
where $\varphi_j(x,y)=O(|(x,y)|^2),\;j=1,2$.
We consider the polynomial $f(z,w)=w$. There is a small neighbourhood $\om(\dl)$ of $0 \in \cplx_z$ such that
\[
 f^{-1}\{t\}\cap (\wtil{S_1}(\dl) \cup \wtil{S_2}(\dl)) =\mathscr{K}^t_1\cup \mathscr{K}^t_2
\]
where 
\begin{align}
\mathscr{K}^t_1&:= \{(x+iy,t)\;:\; ax+B_1y+\varphi_1(x,y)=t, (x,y)\in \om(\dl)\}, \notag\\ 
\mathscr{K}^t_2&:= \{(x+iy,t)\;:\; ax+B_2y+\varphi_2(x,y)=t, (x,y)\in \om(\dl)\}.\notag
\end{align}
Note that if $\varphi_1=0$ and $\varphi_2 =0$, then the above union would have been a union of two line segments,
which is polynomially convex. Without loss of generality, we may take $B_1 \neq 0$. Then, 
$\pi_z(\mathscr{K}^t_1)$ is the graph of the function $\psi_1: \om(\dl)\cap \rea_x \mapp \rea$ with 
$\frac{d\psi_1}{dx}(0)=-a/B_1$. On the other hand:
\smallskip

$\pi_z(\mathscr{K}^t_2)= \begin{cases}
                &\text{the graph of a function}\; \psi_2: \om(\dl)\cap \rea_y \mapp \rea, \;\text{if}\; B_2=0,\\
                & \text{the graph of a function}\; \wtil{\psi_2}:\om(\dl)\cap \rea_x \mapp \rea \;\text{with}\;
                 \frac{d\wtil{\psi_2}}{dx}(0)=-a/B_2,\\
                & \qquad\qquad\qquad\qquad\qquad\qquad\qquad\qquad\qquad\qquad\qquad\;\text{if}\; B_2 \neq 0,
                \end{cases}$

\noindent for $\om(\dl)$ sufficiently small, and viewing $\cplx\cong \rea_x \times \rea_y$. Here is a brief 
justification of the above descriptions of $\pi_z(\mathscr{K}^t_j),~j=1,2$. Note that the equation 
$ax+\varphi_2(x,0)=t$ will have a unique solution in $\om(\dl)\cap \rea_x$, say $x=x_0(t)$, once we have chosen 
a $\dl>0$ sufficiently small and fixed it, for all $t \in \rea$ approaching to $0$. By the Implicit Function 
Theorem, $\psi_2$ is a function satisfying $\psi_2(0)=x_0(t)$ and 
$\frac{d\psi_2}{dy}(0)=-\bdy_y \varphi_2(x_0(t),0)/(a+\bdy_x \varphi_2(x_0(t),0))$. A similar,
but easier, argument gives the descriptions of $\psi_1$ and $\wtil{\psi_2}$. 
\smallskip

In either case, $\pi_z(\mathscr{K}^t_1)\cap \pi_z(\mathscr{K}^t_2)$ does not separate $\cplx_z \times \{0\}$,
whence $\mathscr{K}^t_1 \cup \mathscr{K}^t_2$ does not separate $\cplx_z \times \{t\}$, provided we choose and fix
 $\dl>0$ sufficiently small. In view of Result \ref{R:S1}, $\wtil{S_1}(\dl) \cup \wtil{S_2}(\dl)$ 
is polynomially convex. As $\Psi$ is a biholomorphism, we infer that $S_1\cup S_2$ is locally polynomially 
convex at $(0,0)\in \CC$.
\medskip

\noindent {{\bf {\em Case II}.} {\em When $\Phi(T_0S_1 \cap T_0S_2)= \{(0,u)\in \CC\;:\;u \in \rea\}$. }
\smallskip

\noindent Since $\Phi(T_0S_1\cap T_0S_2)= \{(0,u)\in \CC\;:\;u \in \rea\}$, both the planes $T_0S_1$ and $T_0S_2$
are perpendicular to $\cplx_z \times \{0\}$ in $\CC$. We can find an angle $\tht$ such that if we define
$\Psi(z,w):= (e^{i\tht}\Phi_1,\Phi_2 )$ then, neither  $\pi_z\circ\Psi(T_0S_1)$ nor $\pi_z \circ\Psi(T_0S_2)$ 
is the $x$-axis or the $y$-axis.
Hence we have:
\smallskip

Equation of $\Psi(T_0S_j)=
 \begin{cases}
 y&= A_jx, \\
 v&=0,\; A_j \in \rea \setminus \{0\},\;j=1,2, \; \text{and}\; A_1 \neq A_2. 
 \end{cases} $
\smallskip

For sufficiently small $\dl>0$,
\begin{align}
 \wtil{S_j}(\dl):= \Psi(S_j)\cap \ba{B(0,\dl)}& = 
\begin{cases}
 y&= A_jx+\varphi_j(x,u), \\
 v&=0,
\end{cases}  \notag
\end{align}
where $\varphi_j(x,u)=O(|(x,u)|^2),\;j=1,2$. As in the first case, we consider the polynomial $f(z,w)=w$.
There is a small neighbourhood $\om$ of $0 \in \cplx_z$ such that
\[
 f^{-1}\{t\}\cap (\wtil{S_1}(\dl) \cup \wtil{S_2}(\dl)) =\mathscr{K}^t_1\cup \mathscr{K}^t_2
\]
where 
\begin{align}
\mathscr{K}^t_1&:= \{(x+iy,t)\;:\; y= A_1x+\varphi_1(x,t), (x,y)\in \om(\dl)\} \notag\\ 
\mathscr{K}^t_2&:= \{(x+iy,t)\;:\; y=A_2x+\varphi_2(x,t), (x,y)\in \om(\dl)\}.\notag
\end{align}
Note here also that if $\varphi_1=0$ and $\varphi_2 =0$, then the above union would have been a union of two line segments,
which is polynomially convex. In this case $\pi_z(\mathscr{K}^t_j)$ is the graph of the function 
$\psi_j: \om(\dl)\cap \rea_x \mapp \rea$ with $\frac{d\psi_j}{dx}(0)=A_j$ for $j=1,2$. Hence, 
$\mathscr{K}^t_1 \cup \mathscr{K}^t_2$ does not separate $\cplx_z \times \{t\}$, provided we 
choose and fix $\dl>0$ sufficiently small. Hence, in view of Result 
\ref{R:S1}, $\wtil{S_1}(\dl) \cup \wtil{S_2}(\dl)$ is polynomially convex. As $\Psi$ is a biholomorphism,
 we infer that $S_1\cup S_2$ is locally polynomially convex at $(0,0)\in \CC$. \qed               
\medskip

\section{The Proof of Theorem~\ref{T:linePCVX}}\label{S:proof-linePCVX}
We begin by observing that since $T_0S_1\neq T_0S_2, \ \lambda \neq 0$. We shall use Kallin's lemma with
 the following polynomial
\[ P(z,w)= z+w+ \albar z^2+ \al w^2
,\]
where $\al \in \cplx$ will be chosen suitably. First, we examine the image of $S_1 \cap U$ under $P$. Let us designate
\[
\phi_1(z):= \ba{A} z^2+A\ba{z}^2+C_1z \ba{z}+O(|z|^3), \;\; z \in D(0;\dl).
\]
Thus we have,
\begin{align}
P(z, \ba{z}+\phi_1(z))&=z+\ba{z}+A\ba{z}^2+\ba{A}z^2+C_1|z|^2+\ba{\al}z^2+\al \ba{z}^2
+O(|z|^3), \notag \\
\imag{P(z,\ba{z}+\phi_1(z))}&=\imag{C_1}|z|^2+ O(|z|^3) \;\; \forall z\in D(0;\dl). \label{E:imM1}
\end{align}
Consequently, by using condition $(i)$, we can find a $\del{1} \in (0, \dl)$ sufficiently small so that
\begin{equation} \label{E:pinverselinePCVX}
 P^{-1}\{0\}\cap S_1(\del{1})=\{0\},
\end{equation}
where $S_1(\del{1})=S_1 \cap \ba{D(0;\del{1})}\times \cplx$.
\smallskip

Now let us look at the image of $S_2(\del{1})~(= S_2 \cap \ba{D(0;\del{1})}\times \cplx)$ under the polynomial $P$.
\begin{align}
&P(z,\ba{z}+\lm \ba{z}+\ba{\lm}z+\phi_2(z))\notag\\
&= z+\ba{z}+\lm \ba{z}+\ba{\lm}z+\phi_2(z)+\ba{\al}z^2+\al(\ba{z}+\lm \ba{z}+\ba{\lm}z)^2+O(|z|^3)\notag\\
&= z+\ba{z}+\lm \ba{z}+\ba{\lm}z+A_2 z^2+B_2 \ba{z}^2+ C_2|z|^2+\ba{\al}z^2+\al \ba{z}^2+
2\al \ba{z}(\lm \ba{z}+\ba{\lm}z)\notag\\
& \qquad\qquad\qquad\qquad\qquad +\al (\lm \ba{z}+\ba{\lm}z)^2+O(|z|^3) \notag\\
&=z+\ba{z}+\lm \ba{z}+\ba{\lm}z+A_2 z^2+(B_2+2\al \lm) \ba{z}^2+ (C_2+2\al\ba{\lm})|z|^2+\ba{\al}z^2+\al \ba{z}^2 \notag \\
&\qquad\qquad\qquad\qquad\qquad + \al (\lm \ba{z}+\ba{\lm}z)^2+O(|z|^3) \notag
\end{align}
We choose $\al$ such that
\[
\ba{A_2}=B_2+2 \al \lm.
\]
Note that
\[ C_2+ 2 \al \ba{\lm}= C_2+ \frac{(\ba{A_2}-B_2)\ba{\lm}^2}{|\lm|^2},
\]
and observe:
\begin{multline}
\imag{(P(z, \ba{z}+\ba{\lm}z+ \lm \ba{z}+ \phi_2(z)))}=\imag \left (\frac{(\overline{A_2}-B_2)\lbar^2}{|\lambda|^2}+ C_2\right)|z|^2 \\
+ \imag{\left(\frac{(\ba{A_2}-B_2)\ba{\lm}}{2|\lm|^2}\right)}(\ba{\lm}z+ \lm \ba{z})^2 +O(|z|^3) \;\;
\forall z \in D(0;\dl).\label{E:v1ptnontrans}
\end{multline}
We examine the second term on the right hand side of \eqref{E:v1ptnontrans}:
\begin{align}
\imag{\left(\frac{(\ba{A_2}-B_2) \ba{\lm}}{2|\lm|^2}\right)}\geq 0
 \impl \imag{\left(\frac{(\ba{A_2}-B_2)\ba{\lm}}{2|\lm|^2}\right)}(\ba{\lm}z+ \lm \ba{z})^2
&\leq 2 \imag((\ba{A_2}-B_2)\ba{\lm})|z|^2,
\label{E:imagalphag}
\end{align}
and
\begin{align}
\imag{\left(\frac{(\ba{A_2}-B_2)\ba{\lm}}{2|\lm|^2}\right)}< 0 \impl \imag{\left(\frac{(\ba{A_2}-B_2)\ba{\lm}}{2|\lm|^2}\right)}(\ba{\lm}z+ \lm \ba{z})^2\geq 2 \imag((\ba{A_2}-B_2)\ba{\lm})|z|^2.
\label{E:imagalphal}
\end{align}
We shall divide the remaining part of the proof into two cases.
\smallskip

\noindent{{\bf {\em Case I}.} {\em We consider the case when $\imag(C_1)<0$.}}

\noindent{So, ${\rm sgn}(\imag(C_1))=-1$ and hence by condition $(i)$
\[
\imag \left (\frac{(\overline{A_2}-B_2)\ba{\lm}^2}{|\lambda|^2}+ C_2\right) >0.
\]
If $\imag((\ba{A_2}-B_2)\ba{\lm}) \geq 0$ then by \eqref{E:imagalphag} there is a $\del{2} \in (0,\del{1})$ such that when
$z \in D(0;\del{2})$,
\begin{equation}\label{E:imM21}
 \imag(P(z, \ba{z}+\ba{\lm}z+ \lm \ba{z}+\phi_2(z))) \geq 0
\end{equation}
and equalling $0$ if and only if $z=0$. On the other hand, if $\imag((\ba{A_2}-B_2)\ba{\lm}) < 0 $, then
by \eqref{E:imagalphal}
\begin{align}
&\imag{\left(P(z,\ba{z}+\ba{\lm}z+\lm \ba{z}+\phi_2(z))\right)}\notag\\
& \geq \imag \left (\frac{(\overline{A_2}-B_2)\ba{\lm}^2}{|\lambda|^2}+ C_2\right)|z|^2
+ 2 \imag((\ba{A_2}-B_2)\ba{\lm})|z|^2+O(|z|^3)\notag\\
&=\left(\imag \left (\frac{(\overline{A_2}-B_2)\ba{\lm}^2}{|\lambda|^2}
+ C_2\right)+ 2 \imag((\ba{A_2}-B_2)\ba{\lm})\right)|z|^2+O(|z|^3).\notag
\end{align}
Hence by condition $(ii)$, and arguing exactly as above, we get that there is a 
$\del{2}\in (0,\del{1})$ such that when $z \in D(0;\del{2}) $,
\begin{equation}\label{E:imM22}
 \imag(P(z, \ba{z}+\ba{\lm}z+ \lm \ba{z}+\phi_2(z))) \geq 0
\end{equation}
and equalling $0$ if and only if $z=0$.}
\smallskip

Hence from \eqref{E:imM1},\eqref{E:imM21} and \eqref{E:imM22},
we have the following:
\smallskip

{\em There exists $\del{2}>0$ such that
\begin{itemize}
\item{$P^{-1}\{0\}\cap S_2(\del{2})=\{0\}$; and}
\item{$P(S_1(\del{2}))$ and $P(S_2(\del{2}))$ lie in the lower and upper half planes respectively and 
intersect only at the origin.}
\end{itemize}}
\medskip

\noindent{{\bf {\em Case II}.} {\em We consider the case when $\imag(C_1)>0$.}}

\noindent{Then by condition $(i)$,
\[
\imag \left (\frac{(\overline{A_2}-B_2)\ba{\lm}^2}{|\lambda|^2}+ C_2\right) <0.
\]
We argue similarly as in case \noindent{{\em Case I}. If $\imag((\ba{A_2}-B_2)\ba{\lm}) < 0$ then
there is a $\del{2} \in (0,\del{1})$ such that when $z \in D(0;\del{2})$,
\begin{equation}\label{E:imM23}
 \imag(P(z, \ba{z}+\ba{\lm}z+ \lm \ba{z}+\phi_2(z))) \leq 0
\end{equation}
and equalling $0$ if and only if $z=0$. On the other hand, if $\imag((\ba{A_2}-B_2)\ba{\lm}) \geq 0 $, then:
\begin{align}
&\imag{\left(P(z,\ba{z}+\ba{\lm}z+\lm \ba{z}+\phi_2(z))\right)}\notag\\
& \leq -\left|\imag \left(\frac{(\overline{A_2}-B_2)\ba{\lm}^2}{|\lambda|^2}+ C_2\right)\right| |z|^2+ 2 \imag((\ba{A_2}-B_2)\ba{\lm})|z|^2+O(|z|^3)\notag\\
&=\left(-\left|\imag \left(\frac{(\overline{A_2}-B_2)\ba{\lm}^2}{|\lambda|^2}+ C_2\right)\right|+ 2 \imag\left((\ba{A_2}-B_2)\ba{\lm}\right)\right)|z|^2+O(|z|^3).
\label{E:imntrans}
\end{align}
Hence by condition $(ii)$ and \eqref{E:imntrans}, there is a $\del{2}\in (0,\dl)$ such that when $z \in D(0;\del{2}) $,
\begin{equation}\label{E:imM24}
 \imag(P(z, \ba{z}+\ba{\lm}z+ \lm \ba{z}+\phi_2(z))) \leq 0
\end{equation}
and equalling $0$ if and only if $z=0$.}
\smallskip

In this case also from \eqref{E:imM1}, \eqref{E:imM23} and \eqref{E:imM24} we have the following:
\smallskip

{\em There exists $\del{2}>0$ such that
\begin{itemize}
\item{$P^{-1}\{0\}\cap S_2(\del{2})=\{0\}$; and}
\item{$P(S_1(\del{2}))$ and $P(S_2(\del{2}))$ lie in the upper and lower half planes respectively
and intersect only at the origin.}
\end{itemize}}
\medskip

 Therefore, with this choice of $P$, in all cases, the hypotheses of Kallin's lemma (Lemma \ref{L:kallin}) are met
  and hence $S_1(\del{2})\cup S_2(\del{2})$ is polynomially convex. Hence $S_1\cup S_2$ is locally polynomially 
convex at the origin. \qed
 \medskip

{\bf Acknowledgement.} I am grateful to Gautam Bharali for many useful discussions that we had during the
course of this work. I also wish to thank Nikolay Shcherbina for his helpful comments, especially concerning
an earlier version of Theorem \ref{T:epsperturb}.

\end{document}